\documentclass[10pt]{amsart}
\usepackage[foot]{amsaddr}
\usepackage{bbm}

\author{Robert Bogucki}
\address{Institute of Mathematics, University of Warsaw, ul. Banacha 2, 02-097 Warszawa, Poland}
\email{r.a.bogucki@gmail.com}

\newtheorem{thm}{Theorem}[section]

\newtheorem{wn}[thm]{Corollary}
\newtheorem{lem}[thm]{Lemma}

\theoremstyle{definition}
\newtheorem*{defi}{Definition}
\newtheorem{rem}{Remark}[section]

\theoremstyle{remark}

\def\g{\gamma}

\def\E{\mathbf{E}}
\def\P{\mathbf{P}}
\def\R{\mathbb{R}}

\def\1{\mathbbm{1}}

\newcommand{\A}{\mathcal{A}}
\newcommand{\B}{\mathcal{B}}
\newcommand{\C}{\mathcal{C}}

\def\be#1\ee{\begin{equation}#1\end{equation}}
\newcommand{\ba}{\begin{eqnarray*} }
\newcommand{\ea}{\end{eqnarray*} }

\date{}

\title{Suprema of canonical Weibull processes}
\begin{document}
\begin{abstract}
In this note we investigate the problem of bounding the suprema of canonical processes based on r.v.s with tails $\exp{(-t^r)}$, where $0<r \le 1$. We propose a method using non-increasing rearrangement that provides a two-sided bound.
\end{abstract}

\maketitle

\section*{Introduction}
It is quite common, both in theoretical and practical studies concerning randomness, to ask about the supremum of a stochastic process. More precisely, one is usually interested in the mean of a supremum of a stochastic process over some set. Knowing, that this quantity can be inherently too complicated, we seek for two-sided bounds that differ only by a multiplicative constant. The so-called chaining method, is the modern approach to this problem. Invented by A. Kolmogorov it has proven very useful in answering a number of questions. See monograph \cite{TALBOOK}.

In this paper, we are interested in the study of stochastic processes of the form $X_t = \sum_{k=1}^{\infty} t_k X_k$, where $X_k$ are independent r.v.s. Such processes are called canonical and we are interested in the quantity $\E \sup_{t \in T} X_t$, where $T \subset \ell^2$. To ensure that this series converges a.s. for $t \in \ell^2$, it is enough to assume that $X_k$ are standardized, i.e. have zero mean and unit variance. To avoid any measurability questions and problems, we define
\[
\E \sup_{t \in T} X_t := \sup \left\{ \E \sup_{t \in S} X_t \mid S \subset T \ \text{finite} \right\}.
\]
Consequently, as long as our bounds do not depend on $n$, it is enough to treat the case $T \subset \R^n$ finite.
It turns out, that the road to two-sided bounds for canonical processes leads through the metric space $(T,d)$, where $d$ is some metric related to the process $(X_t)_{t \in T}$. 

Before moving further, let us introduce some notation and recall the basic definitions. For a random variable $X$ and $1 \le p < \infty$, we define $\| X \|_p := \left( \E |X|^p \right)^{1/p}$. Similarly, for a sequence $x = (x_k)_{k=1}^{\infty}$, we define $\| \cdot \|_p$ as the $\ell^p$ norm, i.e. 
\begin{flalign*}
& \|x\|_p  = \left( \sum_{k=1}^\infty |x_k|^p \right)^{1/p}, \\
& \|x\|_\infty = \sup_{1 \le k < \infty} |x_k|.
\end{flalign*} 
By $d_p$ we will denote the metric induced from the $\ell^p$ norm, and by $\Delta_d(A)$,  the diameter of the set $A$ in the metric $d$. Given a sequence of partitions $(\A_n)_{n\ge0}$ of the set $T$, for $t\in T$ by $A_n(t)$ we will denote the unique set $A_n \in \A_n$, such that $t \in A_n$. For a sequence of r.v.s $(Y_k)_{k=1}^n$, by $(Y_k^*)_{k=1}^n$ we will denote the non-increasing rearrangement of absolute values of $(Y_k)_{k=1}^n$, in other words $Y_k^* = \operatorname*{k-\,max}_i |Y_i|$. Furthermore, $L$ will always denote a constant, which may differ at each occurrence. If we allow the constant to depend on the parameter $r$, we will denote it by $L(r)$. A bit informally, by a tail of a symmetric r.v. $X$ we will mean a tail of $|X|$, i.e. the function
$t\to \Pr(|X|> t)$, $t>0$. Finally, for a natural number $n$, $[n]$ will stand for the set $\{1,..., n\}$.
\begin{defi}
A sequence of partitions $(\A_n)_{n\ge0}$ is called increasing if for every $B \in \A_{n+1}$, there exists $A \in \A_n$, such that $B \subset A$.
\end{defi}

\begin{defi}
An increasing set of partitions $(\A_n)_{n\ge0}$ of $T$ is called admissible if $\A_0 = \{T\}$ and $|\A_n| \le 2^{2^n}$.
\end{defi}

\begin{defi}
For $\alpha > 0$ and a metric space $(T,d)$, we define
\[
\g_\alpha (T,d) = \inf \sup_{t\in T} \sum_{n \ge 0} 2^{n/\alpha} \Delta_d(A_n(t)),
\]
where the infimum runs over all admissible sequences of partitions of the set $T$.
\end{defi}

In the Gaussian case, that is when $X_t = \sum_{k=1}^{\infty} t_k g_k$, where $g_k$ are i.i.d. standard Gaussians, we have the celebrated Fernique-Talagrand majorizing measure theorem.
\begin{thm}[Fernique, Talagrand, \cite{FER, TAL_G}]\label{gauss}
Let $(X_t)_{t \in T}$ be a centered Gaussian process, and $d(s,t) = \sqrt{\E |X_s - X_t|^2}$. Then, there exists an absolute constant $L$, such that
\[
\frac{1}{L} \g_2(T,d) \le \E \sup_{t \in T} X_t \le  L  \g_2(T,d).
\]
\end{thm}
In the case of the canonical processes based on independent symmetric r.v.s with tails $\exp{(-t^r)}$ we have the following theorem due to Talagrand.
\begin{thm}[Talagrand, \cite{TAL_R}]\label{r1}
Let $T \subset \ell^2$, and $(X_k)$ be a sequence of independent symmetric r.v.s with tails $\exp{(-t^r)}$, where $r \in [1,2]$. Then
\[
\frac{1}{L(r)} \left\{ \g_2(T,d_2) + \g_r(T,d_\infty) \right\} \le \E \sup_{t \in T} X_t \le L(r) \left\{ \g_2(T,d_2) + \g_r(T,d_\infty) \right\},
\]
where
\[
X_t =  \sum_{k=1}^\infty t_k X_k.
\]
\end{thm}
In \cite{LAT} this result was later extended by R. Lata{\l}a to the case of symmetric variables with log-concave tails that do not decrease too rapidly. That is $\P(|X_k| \ge t) = \exp{(-N_k(t))}$, where $N_k$ is convex, and there is $\gamma < \infty$ such that $N_k(2t) \le \gamma N_k(t)$ . Slightly more general situation was considered in \cite{TTKOCZ_LAT}.

In this note, we treat the case of log-convex tails. We are interested in the so-called Weibull r.v.s, that have tails $\exp{(-|t|^r)}$, where $0 < r \le 1$. In the first section, we show that the lower bound of Theorem \ref{r1} does not hold in the case of $r < 1$. In the second section we present the main theorem of the note, which is a two-sided bound based on the functional $\g_2$ and a random permutation of the index set.

\section{One-sided bound}
In this section we recall the one-sided upper bound and present the counter-example for the lower bound. Although the proof of the upper bound follows standard argument, we present it for completeness. We start with moment inequalities for sums of independent symmetric r.v.s with tails $\exp{(-t^r)}$.
\begin{thm}[Hitczenko, Montgomery-Smith, Oleszkiewicz \cite{HMSO}]\label{hmso}
Let $(Y_k)$ be a sequence of independent r.v.s with log-convex tails and $p\ge2$. Then, for some constant $L$
\[
\frac{1}{L} \mathcal{M}_{Y,p} \le \left\| \sum_k Y_k \right\|_p \le L \mathcal{M}_{Y,p}, 
\]
where
\[
\mathcal{M}_{Y,p} = \left\{ \left(\sum_k \|Y_k\|_p^p  \right)^{1/p} + \sqrt{p} \left(\sum_k \|Y_k\|_2^2 \right)^{1/2} \right\}.
\]
\end{thm}
\begin{wn}\label{momenty}
Let $(X_k)$ be a sequence of independent symmetric r.v.s with tails $\exp{(-t^r)}$, where $0<r \le2$. Then, for $p\ge2$ and some constant $L(r)$ which depends only on $r$
\[
\| \sum_k t_k X_k \|_p \le L(r) \left\{   \sqrt{p} \|t\|_2 + p^{1/r}\|t\|_\infty \right\}.
\]
\end{wn}
\begin{proof}
We will first show that
\[
\|X_k\|_p < L(r) p^{1/r}.
\]
We have $\P(|X_k| > u) = \exp{(-u^r)}$, therefore for $t \ge 1$
\[
\P(\exp{(|X_k|^r/2)} > t) = \frac{1}{t^2}.
\]
Hence
\[
\E \exp{(|X_k|^r/2)} = \int_0^\infty \P(\exp{(|X_k|^r/2)} > t) \ dt \le 1 + \int_1^\infty  \frac{1}{t^2} \ dt = 2.
\]
Using the inequality $(x/a)^a \le \exp{(x)}$ for $x,a\ge0$, we see that
\[
\E \left( \frac{|X_k|^r}{2a}\right)^a \le \E \exp{(|X_k|^r/2)} \le 2.
\]
Setting $p=ra$, we obtain
\[
\|X_k\|_p = (\E |X_k|^p)^{1/p} \le 2^{1/p} (2p/r)^{1/r} \le L(r) p^{1/r}.
\]
We now use Theorem \ref{hmso} for the sequence $(t_kX_k)$ and $p\ge2$ to get
\[
\| \sum_k t_k X_k \|_p \le L(r) \left\{   \sqrt{p} \|t\|_2 + p^{1/r}\|t\|_p \right\}.
\]
To complete the proof it is now enough to show that
\[
\sqrt{p} \|t\|_2 + p^{1/r}\|t\|_p \le L(r) \left\{ \sqrt{p} \|t\|_2 + p^{1/r}\|t\|_\infty \right\}.
\]
Using homogeneity, we can assume that
\[
\sqrt{p} \|t\|_2 + p^{1/r}\|t\|_\infty = 1,
\]
then $\|t\|_2 \le p^{-1/2}$ and $\|t\|_\infty \le p^{-1/r}$. Therefore, for $p\ge2$,
\begin{align*}
\|t\|_p &\le \left( \sum_k |t_k|^2 \|t\|_{\infty}^{p-2} \right)^{1/p} \le \left( p^{-1 - (p-2)/r} \right)^{1/p} 
= \left( p^{-p/r} p^{(2-r)/r}  \right)^{1/p} \\
& \le \frac{e^{(2-r)/r} }{p^{1/r}} 
= \frac{e^{(2-r)/r}}{p^{1/r}} \left\{ \sqrt{p} \|t\|_2 + p^{1/r}\|t\|_\infty \right\},
\end{align*}
where the last inequality follows form the fact that $p^{1/p} \le e$.
\end{proof}

\begin{rem}
For a symmetric r.v. $X$ with tail $\exp{(-t^r)}$ and $t=p^{1/r}$, we have
\[
\E |X|^p \ge t^p \P(|X|\ge t) = p^{p/r} \frac{1}{e^p},
\]
therefore $\|X\|_p \ge L p^{1/r}$ and consequently the bound of the above corollary can be reversed.
\end{rem}

To connect the supremum of the process to its moments one needs the following Lemma due to R. Lata{\l}a and S. Mendelson.

\begin{lem}[Lata{\l}a, Mendelson \cite{TTKOCZ_LAT, MEN}, {\cite[Exercise 2.2.24]{TALBOOK}}]\label{lm}
Let $(\A_n)_{n \ge 0}$ be an arbitrary admissible sequence of partitions of $T$. Then, for every separable process $(X_t)_{t \in T}$
\[
\E \sup_{s,t \in T} |X_s - X_t| \le L \sup_{t \in T} \sum_{k \ge 0} \Delta_k (A_k (t) ),
\]
where
\[
\Delta_k (A) = \sup_{s,t \in A} \|X_s - X_t\|_{2^k} .
\]
\end{lem}

We can now prove the upper bound, see Theorem 2.2.22 in \cite{TALBOOK} for a similar argument.

\begin{thm}\label{gora1}
Let $T \subset \R^n$, and $(X_k)$ be a sequence of independent symmetric r.v.s with tails $\exp{(-t^r)}$, where $0 < r \le 2$. Then, for some constant $L(r)$, which depends only on $r$
\[
\E \sup_{t \in T} X_t \le L(r) \left\{ \g_2(T,d_2) + \g_r(T,d_\infty) \right\},
\]
where
\[
X_t =  \sum_{k=1}^n t_k X_k.
\]
\end{thm}
\begin{proof}
Take an arbitrary $t_0 \in T$, we have $\E X_{t_0} = 0 $ and therefore
\[
\E \sup_{t \in T} X_t = \E \sup_{t \in T} (X_t - X_{t_0}) \le \E \sup_{s,t \in T} |X_t - X_s|.
\]
To use Lemma \ref{lm}, we need an appropriate, admissible sequence of partitions. To create one, we can take two sequences of partitions such that they almost maximize the functionals $\g_2(T,d_2)$, $\g_r(T,d_\infty)$, and then intersect the sets on a given level. 

Let $\A_n$, $\B_n$ be an admissible sequences of partitions such that
\begin{align*}
& \sup_t \sum_{k\ge0} 2^{n/2} \Delta_{d_2}(A_n(t)) \le 2 \g_2(T,d_2), \\
& \sup_t \sum_{k\ge0} 2^{n/r} \Delta_{d_\infty}(B_n(t)) \le 2 \g_r(T,d_\infty).
\end{align*}
We will now construct an admissible sequence of partitions $\C_n$ with the desired properties. Set $\C_0 = \{ T \}$ and
\[
\C_n = \{ A \cap B \mid A \in \A_{n-1}, B \in \B_{n-1} \} \text{\quad for $n \ge 1$.}
\]
Note that $|\C_n| \le |\A_{n-1}||\B_{n-1}| \le 2^{2^n}$. Using Corollary \ref{momenty} with $p=2^n$, we see that  for $t \in T$
\begin{align*}
\Delta_n(C_n(t)) &=  \sup_{s,u \in C_n(t)} \|X_s - X_u\|_{2^n} \le L(r) \sup_{s,u \in C_n(t)}  \left(2^{n/2}\|s-u\|_2 + 2^{n/r}\|s-u\|_\infty \right) \\
& \le L(r) \left( 2^{n/2} \Delta_{d_2}(A_{n-1}(t)) + 2^{n/r}\Delta_{d_\infty}(B_{n-1}(t)) \right) \\
& \le L(r) \left( 2^{(n-1)/2} \Delta_{d_2}(A_{n-1}(t)) + 2^{(n-1)/r}\Delta_{d_\infty}(B_{n-1}(t)) \right),
\end{align*}
where in the penultimate inequality we use the fact that $C_n(t) \subset A_{n-1}(t)$ and $C_n(t) \subset B_{n-1}(t)$. It is now enough to use Lemma \ref{lm} for $\C_n$ to see that
\[
\E \sup_{t \in T} X_t \le L \sup_{t \in T} \sum_{n \ge 0} \Delta_n (C_n (t) ) 
\le L(r) \left\{ \g_2(T,d_2) + \g_r(T,d_\infty) \right\}.
\]

\end{proof}

\subsection{Counter-example for the lower bound}
We will now show, that one cannot reverse the bound of Theorem \ref{gora1} in the case of $r < 1$. The main reason for this is that $\g_r(T,d_\infty)$ is too large compared to the size of the accompanying process.

Let $T = \{ -1;1\}^n$, and $(X_k)$ be a sequence of independent symmetric r.v.s with tails $\exp{(-t^r)}$. We have
\[
\E \sup_{t \in T} \sum_{k=1}^n t_k X_k = \E \sum_{k=1}^n |X_k| = n \E |X_1|.
\]
On the other side, for $k = \lfloor \frac{r+1}{2}\log_2{n}\rfloor < \log_2{n}$, we have $2^n > 2^{2^k}$. Therefore, in every admissible sequence of partitions of $T$, there is a set on the $k$-th level, that contains at least two points. The diameter of $T$ in the norm $\| \cdot \|_\infty$ is exactly $2$. For sufficiently large $n$, we have
\[
\g_r(T,d_\infty) > 2 \cdot 2^{k/r} > 2^{1-1/r} n^{\frac{r+1}{2r}},
\]
however $\frac{r+1}{2r} > 1$, and therefore
\[
\lim_{n \to +\infty} \frac{\g_r(T,d_\infty)}{\E \sup_{t \in T} \sum_{k=1}^n t_k X_k} = +\infty.
\]
In other words, we cannot find a constant $L(r)$, which depends only on $r$, such that the inequality
\[
\g_r(T,d_\infty) \le L(r) \E \sup_{t \in T} \sum_{k=1}^n t_k X_k
\]
holds in the setting of Theorem \ref{gora1}.

\section{Two-sided bound using the non-increasing rearrangement method}
In this section we present a two-sided bound. The method is based on the non-increasing rearrangement of $X_k$. The main result is the following theorem.
\begin{thm}\label{perm}
Suppose that $T \subset \R^n$, $n\ge2$, $X_t = \sum_{k=1}^n t_k X_k$, where $(X_k)$ is a sequence of independent, symmetric r.v.s with tails $\exp{(-t^r)}$. Let $\frac{1}{s} + \frac{1}{2} = \frac{1}{r}$, where $0 < r < 2$. Then for some constant $L(r)$ that depends only on $r$
\[
\frac{1}{L(r)} \E_\pi \g_2(T_\pi) \le \E \sup_{t \in T} X_t \le L(r) \E_\pi \g_2(T_\pi),
\]
where
\[
T_\pi = \left\{ \left( t_{\pi(k)} \left(\log{\frac{n}{k}}\right)^{1/s} \right)_{k=1}^n  \mid  (t_1,...,t_n) \in T \right\},
\]
and $\pi$ is a random (uniformly distributed) permutation of $[n]$.
\end{thm} 
The proof is done in three steps. First, we introduce the conditional Gaussian representation, which is crucial. Then we move to considering upper and lower bounds separately.
\subsection{Conditionally Gaussian representation}
It turns out that the investigated process can be represented using r.v.s which are conditionally Gaussian. To obtain this representation, we need the following key lemma.
\begin{lem}\label{replem}
Let $X$ be a symmetric random variable with tail $\exp{(-t^r)}$, $g$ a standard Gaussian variable, and $Y$ symmetric random variable with tail $\exp{(-t^s)}$, where $\frac{1}{s} + \frac{1}{2} = \frac{1}{r}$ and $0< r < 2$. We can then define a probability space and copies of variables $X,Y,g$ for which $g$ is independent of $Y$ and
\begin{align*}
|gY| &\le L(r) (1 + |X|), \\
|X| & \le L(r) (1+ |gY|).
\end{align*}
\end{lem}
\begin{proof}
For $t > 0$, $\lambda \in [0,1]$ we have
\[
\P(|g| \ge t^\lambda)\P(|Y| \ge t^{(1-\lambda)}) \le \P(|gY| \ge t) \le \P(|g| \ge t^\lambda) + \P(|Y| \ge t^{(1-\lambda)}).
\]
For $t\ge2$ we have
\[
\exp{(-t^2)} \le \P(|g| \ge t) \le \frac{1}{2} \exp{(-t^2/2)}.
\]
Therefore, for $t\ge 2$,
\[
\exp{(-t^{2\lambda})}\exp{(-t^{(1-\lambda)s})} \le \P(|gY| \ge t) \le \frac{1}{2} \exp{(-t^{2\lambda}/2)} + \exp{(-t^{(1-\lambda)s})}.
\]
Take $\lambda = \frac{r}{2}$, then
\[
\exp{(-2t^r)} \le \P(|gY| \ge t) \le 2 \exp{(-t^r/2)}
\]
to finish the proof, we inverse the CDF.
\end{proof}

Let us recall the contraction principle, which will allow us to leverage point inequalities.
\begin{lem}[The Contraction Principle, {\cite[Theorem 4.4]{LT}}]\label{kontrakcja}
Suppose that $T \subset \R^n$ and $|a_k| \le |b_k|$ for $k=1,2,...,n$. Then
\[
\E \sup_{t \in T} \sum_{k=1}^n a_k t_k Z_k \le \E \sup_{t \in T} \sum_{k=1}^n b_k t_k Z_k,
\]
for every sequence of symmetric, independent r.v.s $(Z_k)$.
\end{lem}
The theorem below introduces the aforementioned representation.
\begin{thm}\label{rep}
Suppose that $T \subset \R^n$ and $X_t = \sum_{k=1}^n t_k X_k$, where $(X_k)$ is a sequence of independent, symmetric r.v.s with tails
$\exp{(-t^r)}$ for $0 < r < 2$. Let $\frac{1}{s} + \frac{1}{2} = \frac{1}{r}$ and $(Y_k)$ be a sequence of independent, symmetric r.v.s with tails $\exp{(-t^s)}$. Then, for some constant $L(r)$, which depends only on $r$
\[
\frac{1}{L(r)} \E \sup_{t \in T}  \sum_{k=1}^n t_{\pi(k)} g_{\pi(k)} Y_k^*  \le \E \sup_{t \in T} X_t \le L(r) \E \sup_{t \in T}  \sum_{k=1}^n t_{\pi(k)} g_{\pi(k)} Y_k^*,
\]
where $(g_k)$ is a sequence of standard Gaussian r.v.s, $\pi$ a random permutation, and $(g_k)$, $\pi$, $(Y_k)$ are independent.
\end{thm}
\begin{proof}
Note first that for a sequence of independent Rademacher variables $(\epsilon_k)$, we have
\[
 \E \sup_{t \in T}  \sum_{k=1}^n t_{\pi(k)} g_{\pi(k)} Y_k^* =  \E \sup_{t \in T}  \sum_{k=1}^n t_{k} g_k Y_k
= \E \sup_{t \in T}  \sum_{k=1}^n t_{k} \epsilon_k |g_k Y_k|,
\]
\[
\E \sup_{t \in T} X_t = \E \sup_{t \in T}  \sum_{k=1}^n t_{k} X_k 
= \E \sup_{t \in T}  \sum_{k=1}^n t_{k} \epsilon_k |X_k|,
\]
where the first equation follows from the fact that $(Y_1,...,Y_n)$ has the same distribution as $(Y^*_{\pi(1)},...,Y^*_{\pi(n)})$.
Furthermore, using Lemma \ref{replem} we can assume that
\begin{align*}\label{punktowe} \begin{aligned}
|g_kY_k| & \le L(r) (1 + |X_k|),   \\
|X_k| & \le L(r) (1 + |g_kY_k|). 
\end{aligned} \tag{$1$}
\end{align*} \raisetag{1\baselineskip}
Using Jensen's inequality we have
\[
\E_\epsilon \E_X \sup_{t \in T}  \sum_{k=1}^n t_{k} \epsilon_k |X_k| \ge  \E_\epsilon \sup_{t \in T}  \sum_{k=1}^n t_{k} \epsilon_k \E_X |X_k| \ge \frac{1}{L(r)}  \E_\epsilon \sup_{t \in T}  \sum_{k=1}^n t_{k} \epsilon_k. 
\]
Now, we use the contraction principle (Lemma \ref{kontrakcja}) and point inequalities \eqref{punktowe} to see that
\begin{align*}
 \E \sup_{t \in T}  \sum_{k=1}^n t_{\pi(k)} g_{\pi(k)} Y_k^* &= \E \sup_{t \in T}  \sum_{k=1}^n t_{k} \epsilon_k |g_k Y_k| \\
&\le L(r) \left( \E \sup_{t \in T}  \sum_{k=1}^n t_{k} \epsilon_k + \E \sup_{t \in T} \sum_{k=1}^n t_{k} \epsilon_k |X_k| \right)  \\
&\le L(r)  \E \sup_{t \in T} \sum_{k=1}^n t_{k} \epsilon_k |X_k| = \E \sup_{t \in T} X_t.
\end{align*}
The proof of the upper bound is analogous.
\end{proof}

We will also need a lemma that allows us to omit some of the terms in the sum.
\begin{lem}\label{obcinanie}
Let $(a_k)_{k \ge 1}$ be a non-negative, non-increasing sequence and $\theta \in (0,1]$. Then, for $n \ge \frac{2}{\theta}$.
\[
\E \sup_{t \in T} \sum_{k=1}^n t_{\pi(k)} g_{\pi(k)} a_k \le L(\theta) \E \sup_{t \in T} \sum_{k=1}^{\lceil \theta n \rceil} t_{\pi(k)} g_{\pi(k)} a_k 
\]
\end{lem}
\begin{proof}
Let $m = \lceil \theta n \rceil$ and $n = lm + r$, where $0 \le r < m$. For an integer $i=0,1,...,l$ and a permutation $\pi$, let $\pi^{i}$ be a permutation, such that $\pi^i(k) = \pi(k + im)$, where the addition is done modulo $n$. Then $\pi \mapsto \pi^i$ is a bijection, and we have
\begin{align*}
(l+1) \E \sup_{t \in T}  \sum_{k = 1}^{m} t_{\pi(k)} g_{\pi(k)} a_k &= \sum_{i=0}^l \E_\pi \E_g  \sup_{t \in T} \sum_{k = 1}^{m} t_{\pi^i(k)} g_{\pi^i(k)} a_k \\
&= \sum_{i=0}^l \E_\pi \E_g  \sup_{t \in T} \sum_{k = 1}^{m} t_{\pi(k+im)} g_{\pi(k+im)} a_k \\
&= \sum_{i=0}^l \E_\pi \E_g  \sup_{t \in T} \sum_{k = im+1}^{(i+1)m} t_{\pi(k)} g_{\pi(k)} a_{k-im} \\
& \ge \E_\pi \E_g  \sup_{t \in T} \sum_{i=0}^l \sum_{k = im+1}^{(i+1)m} t_{\pi(k)} g_{\pi(k)} a_{k-im} \\ 
&\ge \E_\pi \E_g  \sup_{t \in T} \sum_{k = 1}^{n} t_{\pi(k)} g_{\pi(k)} a_k,  
\end{align*}
where in the last inequality we have used the contraction principle (Lemma \ref{kontrakcja}) and monotonicity of $(a_k)$.
\end{proof}

\subsection{Lower bound} 

Let us first recall Paley-Zygmund inequality.
\begin{lem}[Paley-Zygmund, {\cite[Lemma 0.2.1]{SK}}]
Let $Z$ be a non-negative random variable and $\E Z^2 < \infty$. Then, for $0 < \theta < 1$
\[
\P( Z \ge \theta \E Z) \ge (1-\theta)^2 \frac{(\E Z)^2}{\E Z^2}.
\]
\end{lem}

We are now in a position to prove the lower bound from Theorem \ref{perm}.
\begin{thm}
In the setting of Theorem \ref{perm}, we have
\[
\frac{1}{L(r)} \E_\pi \g_2(T_\pi) \le  \E \sup_{t \in T} X_t.
 \]
\end{thm}
\begin{proof}
Let us first assume that $n \ge 512$. Using Theorem \ref{rep} it is enough to treat the "conditionally Gaussian" case. Our plan is to prove that $Y_k^* >\frac{1}{L(r)} \log{\frac{n}{k}}$ for all $k$ with decent probability. It turns out, that it is enough to control some (rather sparse) subsequence of $Y_k^*$. Let $m$ be such that $2^{2^m} < n \le 2^{2^{m+1}}$. We now set $k_j = \left\lceil \frac{n}{2^{2^j}} \right\rceil$, $\theta_j = 2^{-2^{j-1}}$ and
\[
u_j = \left( \log{ \frac{\theta_j n}{k_j} } \right)^{1/s},
\]
for $j=0,1,2,...,m$. Note, that
\[
k_j = \theta_j n \exp{(-u_j^s)} =  \theta_j \E \sum_{i=1}^n \1_{\{Y_i \ge u_j\}}.
\]
Let $A= \{Y_{k_3}^* \ge u_3, \dots, Y_{k_{m}}^* \ge u_{m}\}$, we have
\[
\P(A) =  1 - \P( \exists_{m  \ge j \ge 3} \ Y_{k_j}^* < u_j) 
\ge 1 - \sum_{j=3}^{m}  \P( Y_{k_j}^* < u_j) \ge 1 - \sum_{j=3}^{m} z_j,
\]
where the last inequality holds, if we assume that $\P( Y_{k_j}^* < u_j) \le z_j$. Equivalently, we need
\[
1 - \P( Y_{k_j}^* \ge u_j) \le z_j.
\]
We will therefore show that there exist $z_3,...,z_m \ge 0$ such that
\[
\P( Y_{k_j}^* \ge u_j) \ge 1-z_j
\]
and
\[
\sum_{j=3}^{m} z_j < \frac{1}{2}.
\]
Using Paley-Zygmund inequality and the definitions of $k_j$, $u_j$ and $\theta_j$, we see that
\begin{align*}
\P(Y_{k_j}^* \ge u_j) &=  \P \left( \sum_{i=1}^n \1_{\{Y_i \ge u_j\}} \ge k_j\right) = \P \left( \sum_{i=1}^n \1_{\{Y_i \ge u_j\}} \ge \theta_j  \E \sum_{i=1}^n \1_{\{Y_i \ge u_j\}} \right) \\
& \ge (1-\theta_j)^2 \frac{(\E\sum_{i=1}^n \1_{\{Y_i \ge u_j\}})^2}{\E(\sum_{i=1}^n \1_{\{Y_i \ge u_j\}})^2\}}  \\
& \ge (1-2\theta_j) \frac{n^2 \exp(-2u_j^s)}{n \exp(-u_j^s) + n^2 \exp(-2u_j^s)}.
\end{align*}
The function $\frac{x}{1+x}$ is increasing. Therefore, to minimize the right hand side, it is enough to minimize $n \exp(-u_j^s)$. Equivalently, we need $u_j^s$ to be as large as possible. The sequence $u_j^s$ is increasing and for $j \le m$, hence
\[
u_j^s \le u_{m}^s = \log{\left(\theta_{m} \frac{n}{k_{m}}\right)} \le \log{\left(2^{-2^{m-1}}2^{2^{m}}\right)} \le \log{\left(2^{2^{m-1}}\right)} \le \log{\sqrt{n}}.
\]
Therefore
\begin{align*}
\P(Y_{k_j}^* \ge u_j) &\ge (1-2\theta_j)\frac{n}{\sqrt{n} + n} =: 1- z_j, \\
z_j &= \left(1-\frac{n}{\sqrt{n} + n}\right) + 2 \theta_j \frac{n}{\sqrt{n} + n} = \frac{1}{\sqrt{n} + 1} + 2 \theta_j \frac{n}{\sqrt{n} + n}.
\end{align*}
We see that
\[
\sum_{j=3}^m \theta_j = \sum_{j=3}^m 2^{-2^{j-1}} < 2 \cdot 2^{-4} = \frac{1}{8},
\]
and finally
\[
\sum_{j=3}^m z_j < \log_2\log_2 n \frac{1}{\sqrt{n} + 1} + 2 \sum_{j=3}^m \theta_j \frac{n}{\sqrt{n} + n} < \frac{1}{2}.
\]
We argue that on the set $A$
\[
Y_k^* \ge \frac{1}{L(r)} \left(\log\frac{n}{k}\right)^{1/s} \text{\quad for $1 \le k \le k_3$.}
\]
To see this, note that 
\[
2\frac{n}{l} \ge \left\lceil \frac{n}{l} \right\rceil \ge \frac{n}{l} \text{\quad for $1 \le l \le n$,}
\]
thus for $3 \le j \le m$,
we have
\[
\frac{\left(\log{\frac{n}{k_j}}\right)^{1/s}}{\left(\log{\frac{\theta_j n}{k_j}}\right)^{1/s}} \le \frac{\left(\log{2^{2^j}}\right)^{1/s}}{\left(\log{2^{-2^{j-1}} \frac{1}{2} 2^{2^j}}\right)^{1/s}}  
=\left( \frac{2^j}{2^{j} - 2^{j-1} - 1} \right)^{1/s} \le 4^{1/s},
\]
and therefore on the set $A$
\[
Y_{k_j}^* \ge \frac{1}{L(r)} \left(\log\frac{n}{k_j}\right)^{1/s}.
\]
For each $1 \le k \le k_3$, we can find the smallest $k_j$, such that $k_j \ge k$, as above we have
\begin{align*}\label{4s}
\frac{\left(\log{\frac{n}{k_j}}\right)^{1/s}}{\left(\log{\frac{n}{k}}\right)^{1/s}} \le
\frac{\left(\log{\frac{n}{k_j}}\right)^{1/s}}{\left(\log{\frac{n}{k_{j-1}}}\right)^{1/s}} \le \left( \frac{2^j}{2^{j-1} - 1} \right)^{1/s} \le 4^{1/s}
\tag{$2$}
\end{align*} 
and consequently on the set $A$
\[
Y_k^* \ge Y_{k_j}^*  \ge \frac{1}{L(r)} \left(\log\frac{n}{k_j}\right)^{1/s} \ge \frac{1}{L(r)} \left(\log\frac{n}{k}\right)^{1/s}.
\]
Using the contraction principle (Lemma \ref{kontrakcja}), we see that on the set $A$
\[
\E_g \sup_{t \in T} \sum_{k=1}^{k_3} t_{\pi(k)} g_{\pi(k)} Y_k^* 
\ge \frac{1}{L(r)} \E_g \sup_{t \in T}  \sum_{k=1}^{k_3} t_{\pi(k)} g_{\pi(k)} \left(\log\frac{n}{k}\right)^{1/s}.
\]
Using the representation given by Theorem \ref{rep} and the fact that $\P(A) > \frac{1}{2}$
\begin{align*}
\E \sup_{t \in T} X_t 
& \ge  \frac{1}{L(r)} \E_Y \E_\pi \E_g \sup_{t \in T}  \sum_{k=1}^n t_{\pi(k)} g_{\pi(k)} Y_k^*  \\
& \ge \frac{1}{L(r)} \E_Y \left[ \1_A  \E_\pi \E_g \sup_{t \in T} \sum_{k=1}^{k_3} t_{\pi(k)} g_{\pi(k)} Y_k^* \right]\\
& \ge \frac{1}{L(r)} \E_\pi \E_g \sup_{t \in T}  \sum_{k=1}^{k_3} t_{\pi(k)} g_{\pi(k)} \left(\log\frac{n}{k}\right)^{1/s}.
\end{align*}
It is now enough to use Lemma \ref{obcinanie} with $\theta = 2^{-2^3}$ and Theorem \ref{gauss}, to see that
\[
\E \sup_{t \in T} X_t \ge \frac{1}{L(r)} \E_\pi \g_2(T_\pi).
\]
If $n < 512$, we see that all $Y_k^*$ are greater than some $c>0$ on a set of measure at least $1/2$. After possibly increasing the constant $L(r)$, the inequality still holds.
\end{proof}

\subsection{Upper bound}
The proof of the upper bound uses similar methods as the previous one, but is a bit easier.
\begin{thm}
In the setting of Theorem \ref{perm}, we have
\[
 \E \sup_{t \in T} X_t \le L(r) \E_\pi \g_2(T_\pi).
 \]
\end{thm}
\begin{proof}
Analogously, as in the case of lower bound, it is enough to treat the case, when $n$ is sufficiently large.
Using symmetry, representation from Theorem \ref{rep} and Lemma \ref{obcinanie} we can write
\[
\E \sup_{t \in T} X_t \le L(r) \E  \sup_{t \in T} \sum_{k=1}^n t_{\pi(k)} \epsilon_k g_{\pi(k)} Y_k^*
\le L(r) \E  \sup_{t \in T} \sum_{k\le n/4} t_{\pi(k)} \epsilon_k g_{\pi(k)} Y_k^*.
\]
Our plan is to integrate by parts, therefore we need to know how to bound the corresponding tails. Note, that using the contraction principle (Lemma \ref{kontrakcja}), for $\tau \ge 0$ we have
\begin{align*}
\P_Y \Bigg( \E_\epsilon \sup_{t \in T}  \sum_{k\le n/4} t_{\pi(k)} \epsilon_k  & g_{\pi(k)} Y_{k}^{*}  \ge 4^{1/s} \tau \E_\epsilon \sup_{t \in T} \sum_{k \le n/4} t_{\pi(k)} \epsilon_k g_{\pi(k)} \left(\log{\frac{n}{k}}\right)^{1/s} \Bigg) \\
 & \le \P \left( \exists_{k\le n/4} \  Y_{k}^* \ge 4^{1/s} \tau \left(\log{\frac{n}{k}}\right)^{1/s} \right).
\end{align*}
Once again, we find $m$, such that $2^{2^m} < n \le 2^{2^{m+1}}$, and set
\[
k_j = \left\lceil \frac{n}{2^{2^j}} \right\rceil, u_j = \left( \log{ \frac{n}{k_j} } \right)^{1/s} \text{\quad for $j=0,1,...,m+1$.}
\]
Consider $k \le n/4$ and take the largest $k_j$ such that $k_j \le k$. Using \eqref{4s} and comparing $u_j$ with $u_{j+1}$, we see that
\[
Y_{k}^* \ge 4^{1/s}\tau\left(\log\frac{n}{k}\right) \implies Y_{k_j}^* \ge \tau\left(\log\frac{n}{k_j}\right),
\]
and therefore
\[
\P\left( \exists_{k\le n/4} \ Y_{k}^* \ge 4^{1/s} \tau \left(\log\frac{n}{k}\right)^{1/s}\right)
\le \P\left( \exists_{1\le j \le m + 1} \  Y_{k_j}^* \ge \tau \left(\log\frac{n}{k_j}\right)^{1/s}\right).
\]
Note that
\begin{align*}
\P(Y_{k_j}^* \ge \tau u_j) &= \P\left( \sum_{i=1}^n \1_{\{Y_i \ge \tau u_j \}} \ge k_j \right) \le
 \frac{\E \sum_{i=1}^n \1_{\{Y_i \ge \tau u_j \}}}{ k_j}  \\
& \le  \frac{n\exp{\left(-(\tau u_j)^s\right)}}{k_j} = \frac{n}{k_j} \left(\frac{k_j}{n}\right)^{\tau^s} =  \left(\frac{k_j}{n}\right)^{\tau^s - 1}.
\end{align*}
We have $\frac{k_j}{n} \le 2 \cdot 2^{-2^j}$ for $j \le m$, and $k_{m+1} = 1$. Therefore, 
\[ 
\left(\frac{k_j}{n}\right)^{\tau^s - 1} \le \left(\frac{1}{2^{2^{j-1}}}\right)^{\tau^s - 1} \text{\quad for $\tau \ge 1$, $1 \le j \le m+1$.}
\]
For $\tau \ge 2^{1/s}$, we now have
\begin{align*}
\P\left( \exists_{1 \le j \le m+1}\  Y_{k_j}^*  \ge \tau \left(\log\frac{n}{k_j}\right)^{1/s}\right) & \le  \sum_{j = 1}^{m+1} \P\left( Y_{k_j}^* \ge \tau \left(\log\frac{n}{k_j}\right)^{1/s}\right) \\
& \le \sum_{j=1}^{m+1} \left(\frac{1}{2^{2^{j-1}}}\right)^{\tau^s-1} \\
& \le 2 \left(\frac{1}{2}\right)^{\tau^s-1} = 2^{-\tau^s+2}.
\end{align*}
Integrating by parts, using the contraction principle (Lemma \ref{kontrakcja}) and Theorem \ref{gauss} we finish the proof
\begin{align*}
\E_{g,\pi} & \E_Y \E_\epsilon \sup_{t \in T} \sum_{k\le n/4} t_{\pi(k)} \epsilon_k g_{\pi(k)} Y_k^*  \\ 
& =  \E_{g,\pi} \int_0^{\infty} \P_Y\left(\E_\epsilon \sup_{t \in T} \sum_{k\le n/4} t_{\pi(k)} \epsilon_k g_{\pi(k)} Y_k^* \ge \tau\right) \ d\tau \\
& \le L(r)  \E_{g,\pi}  \E_\epsilon \sup_{t \in T} \sum_{k \le n/4} t_{\pi(k)} \epsilon_k g_{\pi(k)} \left(\log\frac{n}{k}\right)^{1/s} \left( 2^{1/s}+ \int_{2^{1/s}}^{\infty} 2^{-\tau^s + 2} \ d\tau \right) \\
& \le L(r)   \E_\pi \E_g \E_\epsilon \sup_{t \in T} \sum_{k =1}^n t_{\pi(k)} \epsilon_k g_{\pi(k)} \left(\log\frac{n}{k}\right)^{1/s} \\
& \le L(r)  \E_\pi \g_2(T_\pi).
\end{align*}
\end{proof}

By analyzing the proofs, we see that for $r \ge r_0 >0$, we have $L(r) \le L(r_0) < \infty$, in other words the constant $L(r)$ explodes only when $r \to 0^+$.

\section{Corollaries}
\begin{defi}
We say that $T \subset \R^n$ is permutationally invariant if for every permutation $\pi : [n] \longrightarrow [n]$, we have
\[
\left\{ (t_{\pi(1)},...,t_{\pi(n)}) \mid (t_1,...,t_n) \in T \right\} = T.
\]
\end{defi}
\begin{defi}
For $T \subset \R^n$, we define
\[
T^s = \left\{ \left( t_k \left(\log{\frac{n}{k}}\right)^{1/s} \right)_{k=1}^n \mid (t_1,...,t_n) \in T \right\}.
\]
\end{defi}
The theorem below follows immediately from the above definitions and Theorem \ref{perm}.
\begin{thm}
Suppose that $T \subset \R^n$ is permutationally invariant and $X_t = \sum_{k=1}^n t_k X_k$, where $(X_k)$ is a sequence of independent, symmetric r.v.s with tails $\exp{(-t^r)}$. Let $\frac{1}{s} + \frac{1}{2} = \frac{1}{r}$, where $0 < r \le 2$. Then, for some constant $L(r)$, which depends only on $r$
\[
\frac{1}{L(r)} \g_2(T^s) \le \E \sup_{t \in T} X_t \le L(r) \g_2(T^s).
\]
\end{thm}
For $r \in [1,2]$ we can use Theorem \ref{r1} and Theorem \ref{perm} to obtain results that appeal only to the geometry of the set.
\begin{thm}
Suppose that $T \subset \R^n$ and $\frac{1}{s} + \frac{1}{2} = \frac{1}{r}$, where $r\in [1,2]$. Then, for some constant $L$
\[
\frac{1}{L} \left\{ \g_2(T,d_2) + \g_r(T,d_\infty) \right\} \le \E_\pi \g_2(T_\pi) \le L \left\{ \g_2(T,d_2) + \g_r(T,d_\infty) \right\},
\]
where
\[
T_\pi = \left\{ \left( t_{\pi(k)} \left(\log{\frac{n}{k}}\right)^{1/s} \right)_{k=1}^n  \mid  (t_1,...,t_n) \in T \right\},
\]
and $\pi$ is a random (uniformly distributed) permutation.

\end{thm}
\begin{wn}
Suppose that $T \subset \R^n$ is permutationally invariant and $r \in [1,2]$. Then, for $\frac{1}{s} + \frac{1}{2} = \frac{1}{r}$ and some constant $L$
\[
\frac{1}{L} \left\{ \g_2(T,d_2) + \g_r(T,d_\infty) \right\} \le \g_2(T^s,d_2) \le L \left\{ \g_2(T,d_2) + \g_r(T,d_\infty) \right\}.
\]
\end{wn}

\section{Acknowledgments}

The author would like to express his sincere gratitude to R. Lata{\l}a for encouragement, fruitful discussions and introduction into the topic.

\end{document}